\documentclass[11pt]{article}

\usepackage[leqno]{amsmath}
\usepackage{amssymb}
\usepackage{amsthm}
\usepackage{amsfonts}
\usepackage{epsfig}
\usepackage{latexsym}
\usepackage{url}
\usepackage[pdftex]{hyperref}
\usepackage{ifthen}
\usepackage{comment}

\usepackage{natbib}
 \bibpunct[, ]{(}{)}{,}{a}{}{,}

\newcommand{\field}[1]{\mathbb{#1}}
\newcommand{\R}{\field{R}} 
\DeclareMathOperator{\st}{s.t.}
\newcommand{\Ascr}{{\cal A}}

\newcommand{\Fscr}{{\cal F}}

\newcommand{\Sscr}{{\cal S}}
\newcommand{\Tscr}{{\cal T}}

\newcommand{\Ab}{{\mathbb{A}}}
\newcommand{\Xb}{{\mathbb{X}}}
\newcommand{\Db}{{\mathbb{D}}}

\newcommand\argmax{\mathop{\mbox{{\rm argmax}}}\limits}
\newcommand\argsup{\mathop{\mbox{{\rm argsup}}}\limits}

\newtheorem{theorem}{Theorem}

\newtheorem{lemma}{Lemma}

\newtheorem{property}{Property}

\newtheorem{example}{Example}
\numberwithin{equation}{section}

\usepackage{algorithm}
\usepackage{algorithmic}

\newtheorem*{le:epsclose}{Lemma~\ref{le:epsclose}}
\newtheorem*{le:tvrate}{Lemma~\ref{le:tvrate}}
\newtheorem*{le:bad}{Lemma~\ref{le:bad}}
\newtheorem*{th:zerofrac}{Theorem~\ref{th:zerofrac}}
\newtheorem*{th:opt}{Theorem~\ref{th:opt}}

\newtheorem*{t2}{Theorem 2}
\newtheorem*{l1}{Lemma 1}
\newtheorem*{l2}{Lemma 2}

\newtheorem*{l5}{Lemma 5}

\newcommand{\GETS}{\leftarrow}

\usepackage{setspace}
\title{Irrevocable Multi-Armed Bandit Policies}
\author{Vivek F. Farias \footnote{Sloan School of Management and Operations Research Center, Massachusetts Institute of Technology, email :vivekf@mit.edu}\\
\and
Ritesh Madan \footnote{Qualcomm-Flarion Technologies, email :rkmadan@stanfordalumni.org}\\
}

\begin{document}
\maketitle


\begin{abstract}
This paper considers the multi-armed bandit problem with multiple
simultaneous arm pulls. We develop a new `irrevocable' heuristic for
this problem. In particular, we do not allow recourse to arms that
were pulled at some point in the past but then discarded. This
irrevocable property is highly desirable from a practical
perspective. As a consequence of this property, our heuristic
entails a minimum amount of `exploration'. At the same time, we find
that the price of irrevocability is limited for a broad useful class
of bandits we characterize precisely. This class includes one of the
most common applications of the bandit model, namely, bandits whose
arms are `coins' of unknown biases. Computational experiments with a
generative family of large scale problems within this class indicate
losses of up to $5-10\%$ relative to an upper bound on the
performance of an optimal policy with no restrictions on
exploration. We also provide a worst-case theoretical analysis that
shows that for this class of bandit problems, the price of
irrevocability is uniformly bounded: our heuristic earns expected
rewards that are always within a factor of $1/8$ of an optimal
policy with no restrictions on exploration. In addition to being an
indicator of robustness across all parameter regimes, this analysis
sheds light on the structural properties that afford a low price of
irrevocability.\end{abstract}

\newpage
\pagenumbering {arabic}
\section{Introduction}

Consider the operations of a `fast-fashion' retailer such as Zara or
H\&M. Such retailers have developed and invested in merchandize
procurement strategies that permit lead times for new fashions as
short as two weeks. As a consequence of this flexibility, such
retailers are able to adjust the assortment of products offered on
sale at their stores  to quickly adapt to popular fashion trends. In
particular, such retailers use weekly sales data to refine their
estimates of an item's popularity, and based on such revised
estimates weed out unpopular items, or else re-stock demonstrably
popular ones on a week-by-week basis. In sharp contrast, traditional
retailers such as J.C. Penney or Marks and Spencer face lead times
on the order of several months. As such these retailers need to
predict popular fashions months in advance and are allowed virtually
no changes to their product assortments over the course of a sales
season which is typically several months in length. Understandably,
this approach is not nearly as successful at identifying high
selling fashions and also results in substantial unsold inventories
at the end of a sales season. In view of the great deal of a-priori
uncertainty in the popularity of a new fashion and the speed at
which fashion trends evolve, the fast-fashion operations model is
highly desirable and emerging as the de-facto operations model for
large fashion retailers.

Among other things, the fast-fashion model relies crucially on an
effective technology to learn from purchase data, and adjust product
assortments based on such data. Such a technology must strike a
balance between `exploring' potentially successful products and
`exploiting' products that are demonstrably popular. A convenient
mathematical model within which to design algorithms capable of
accomplishing such a task is that of the multi-armed bandit. While
we defer a precise mathematical discussion to a later section, a
multi-armed bandit consists of multiple (say $n$) `arms', each
corresponding to a \emph{Markov Decision Process}. As a special
case, one may think of each arm as an independent  binomial coin with an
uncertain bias specified via some prior distribution. At each point
in time, one may `pull' up to a certain number of arms (say $k < n$)
simultaneously, or equivalently, toss up to a certain number of
coins. For each tossed coin, we earn a reward proportional to its realization and are able to refine our estimate of its bias based on this realization. We neither learn about, nor earn rewards
from coins that are not tossed. The multi-armed bandit problem
requires finding a policy that adaptively selects $k$ arms to pull
at every point in time with a view to maximizing total expected
reward earned over some finite time horizon or alternatively,
discounted rewards earned over an infinite horizon or perhaps, even long term
average rewards.

With multiple simultaneous pulls allowed, the multi-armed bandit
problem we have described is computationally hard. A popular and
empirically successful heuristic for this problem was proposed
several decades ago by Whittle. Whittle's heuristic produces an
index for every arm based on the state of that arm and simply calls
for pulling the $k$ arms with the highest index at every point in
time. While it has been empirically and computationally observed
that Whittle's heuristic provides excellent performance, the
heuristic typically calls for frequent changes to the set of arms
pulled that might, in hindsight, have been unnecessary. For
instance, in the retail context, such a heuristic may choose to
discard from the assortment a product presently being offered for
sale in favor of a new product whose popularity is not known
precisely. Later, the heuristic may well choose to re-introduce the
discarded product. While such exploration may appear necessary if
one is to discover profitable bandit arms (or popular products),
enabling such a heuristic in practice will typically call for a
great number of adjustments to the product assortment -- a
requirement that is both expensive and undesirable. This begs the
following question: Is it possible to design a heuristic for the
multi-armed bandit problem that comes close to being optimal with a
minimal number of adjustments to the set of arms pulled over time?

This paper introduces a new `irrevocable' heuristic for the
multi-armed bandit problem we call the packing
heuristic. The packing heuristic establishes a static ranking of
bandit arms based on a measure of their potential value relative to the time required to realize that value, and pulls arms in
the order prescribed by this ranking. For an arm currently being
pulled, the heuristic may either choose to continue pulling that arm
in the next time step or else discard the arm in favor of the next
highest ranked arm not currently being pulled. Once discarded, an
arm will \emph{never} be chosen again; hence the term irrevocable. Irrevocability is an attractive structural constraint to impose on arm selection policies in a number of practical applications of the bandit model such as the dynamic assortment problem we have discussed or sequential drug trials where recourse to drugs whose testing was discontinued in the past is socially unacceptable. It is clear that an irrevocable heuristic makes a minimal number of changes to the
set of arms pulled. What is perhaps surprising, is that the
restriction to an irrevocable policy is typically far less expensive than one
might expect. In particular, we demonstrate via a theoretical analysis and computational
experiments that the use of the packing
heuristic incurs a small performance loss relative to  an optimal bandit policy with \emph{no restriction} on exploration, i.e. an optimal strategy that is \emph{allowed recourse} to arms that were pulled but discarded in the past.

More specifically, the present work makes the following
contributions:

\begin{itemize}
\item We introduce a new `irrevocable' heuristic, the packing heuristic, for the multi-armed bandit problem with multiple simultaneous arm-pulls. The packing heuristic is irrevocable in that if an arm being pulled is at some point discarded from the set of arms being pulled, it is never pulled again. At the same time, the performance loss incurred relative to an optimal, potentially non-irrevocable, control policy is limited. In particular, computational experiments with the packing heuristic for a generative family of large scale bandit problems indicate performance losses of up to about a few percent relative to an upper bound on the performance of an optimal policy with no restrictions on exploration. This level of performance suggests that the packing heuristic is likely to serve as a viable heuristic for the multi-armed bandit with multiple plays even when irrevocability is not a concern.

In addition to our computational study, we are able to demonstrate a \emph{uniform} bound on the price of irrevocability for a broad, interesting class of bandits. This class includes most commonly used applications of the bandit model such as bandits whose arms are `coins' of unknown biases. We demonstrate that the packing heuristic earns expected rewards that are always within a factor of $1/8$ of an optimal, potentially non-irrevocable policy. Such a uniform bound guarantees robust performance across all parameter regimes; in particular, the packing heuristic will `track' the performance of an optimal, potentially non-irrevocable policy across all parameter regimes. In addition, our analysis sheds light on the structural properties that afford the surprising efficacy of the irrevocable policies considered here.
\item In the interest of practical applicability, we develop a fast combinatorial implementation of the packing heuristic. Assuming that an individual arm has $O(\Sigma)$ states, and given a time horizon of $T$ steps, optimal solution to the multi-armed bandit problem under consideration requires $O(\Sigma^n T^n)$ computations. The main computational step in the packing heuristic calls for the one time solution of a linear program with $O(n\Sigma T)$ variables, whose solution via a generic LP solver requires $O(n^3\Sigma^3 T^3)$ computations. We develop a novel combinatorial algorithm that solves this linear program in $O(n\Sigma^2T\log T)$ steps by solving a sequence of dynamic programs for each bandit arm. The technique we develop here is potentially of independent interest for the solution of `weakly coupled' optimal control problems with coupling constraints that must be met in expectation. Employing this solution technique, our heuristic requires a total of $O(n\Sigma^2 \log T)$ computations per time step amortized over the time horizon. In comparison, the simplest theoretically sound heuristics in existence for this multi-armed bandit problem (such as Whittle's heuristic) require $O(n\Sigma^2 T)$ computations per time step. As such, we establish that the packing heuristic is computationally attractive.
\end{itemize}

\subsection{Relevant Literature}

The multi-armed bandit problem has a rich history, and a number of excellent references (such as \cite{Gittins89}) provide a thorough treatment of the subject. We review here literature especially relevant to the present work. In the case where
$k=1$, that is, allowing for a single arm to be pulled in a given
time step, \cite{Gittins74} developed an elegant index based policy
that was shown to be optimal for the problem of maximizing
discounted rewards over an infinite horizon. Their index policy is
known to be suboptimal if one is allowed to pull more than a single
arm in a given time step. \cite{Whittle88} developed a simple index
based heuristic for a more general bandit problem (the `restless'
bandit problem) allowing for multiple arms to be pulled in a given
time step. While his original paper was concerned with maximizing
long-term average rewards, his heuristic is easily adapted to other
objectives such as discounted infinite horizon rewards or expected
rewards over a finite horizon (see for instance \cite{Gallien07, BertsimasNinoMora00}).
\cite{Weiss92} subsequently established that under
suitable conditions, Whittle's heuristic was asymptotically optimal
(in a regime where $n$ and $k$ go to infinity keeping $n/k$
constant). Whittle's heuristic may be viewed as a modification to the optimal control policy one obtains upon relaxing the
requirement that at most $k$ arms be pulled in a given time step to
requiring that at most $k$ arms be pulled in expectation in any
given time step. The packing heuristic we introduce is motivated by a similar relaxation. In particular, we restrict attention to policies that entail a total of at most $kT$ arm pulls over the
entire horizon in expectation while allowing for no more than $T$
pulls of any given arm. Where we differ substantially from Whittle's
heuristic is the manner in which we construct a feasible policy (one
where at most $k$ arms are pulled in a given time step) from the
relaxed policy. In fact there are potentially many reasonable ways
of transforming an optimal policy for the relaxed problem to a
feasible policy for the multi-armed bandit; for instance
\cite{BertsimasNinoMora00} use a scheme distinct from both Whittle's and ours,
that employs optimal primal and dual solutions to a linear
programming formulation of Whittle's relaxation to construct an
index heuristic for arm selection. Nonetheless, none of these schemes are irrevocable and nor do they offer non-asymptotic  performance guarantees, if any.

The packing heuristic policy builds upon recent insights on the `adaptivity'
gap for stochastic packing problems. In particular, \cite{GoemansDean00}
recently established that a simple static rule (Smith's rule) for packing a
knapsack with items of fixed reward (known a-priori), but whose
sizes were stochastic and unknown a-priori was within a constant
factor of the optimal adaptive packing policy. \cite{Munagala07} used this
insight to establish a similar static rule for `budgeted learning problems'. In such a problem one is interested in
finding a coin with highest bias from a set of coins of uncertain
bias, assuming one is allowed to toss a \emph{single} coin in a
given time step and that one has a finite budget on the number of such experimental tosses allowed. Our work parallels that work in that we
draw on the insights of the stochastic packing results of
\cite{GoemansDean00}. In addition, we must address two
significant hurdles - correlations between the
total reward earned from pulls of a given arm and the total number
of pulls of that arm (these turn out not to matter in the budgeted
learning setting, but are crucial to our setting), and secondly, the
fact that multiple arms may be pulled simultaneously (only a single
arm may be pulled at any time in the budgeted learning setting).
Finally, a working paper (\cite{Goel07}), brought to our attention by
the authors of that work considers a variant of the budgeted
learning problem of \cite{Munagala07} wherein one is allowed to toss
multiple coins simultaneously. While it is conceivable that their
heuristic may be modified to apply to the multi-armed bandit problem we address, the heuristic
they develop is also \emph{not} irrevocable.

Restricted to coins, our work takes an inherently Bayesian views of
the multi-armed bandit problem. It is worth mentioning that there
are a number of non-parametric formulations to such problems with a vast associated literature. Most relevant to the present model are the papers by \cite{Anantharam1,Anantharam2} that develop simple `regret-optimal' strategies for multi-armed bandit problems with multiple simultaneous plays.

Our development of an irrevocable
policy for the multi-armed bandit problem was originally motivated by
applications of this framework to `dynamic assortment' problems of the type mentioned in the introduction. In
particular, \cite{Gallien07} computationally explore the use of a
number of simple index-type heuristics (similar to Whittle's
heuristic) for such problems, none of which are irrevocable;
nonetheless, they stress the importance of a minimal number of changes to the
assortment if any such heuristic is to be practical.

The remainder of this paper is organized as follows. Section 2 presents the multi-armed bandit model we consider and develops an (intractable) LP whose solution yields an optimal control policy for this bandit problem. Section 3 develops the packing heuristic by considering a suitable relaxation of the multi-armed bandit problem. Section 4 introduces a structural property for bandit arms we call the `decreasing returns' property. It is shown that a useful class of bandits, namely the `coin' bandits relevant to the applications that motivate us, possess this property. That section then establishes that the price of irrevocability for bandits possessing the decreasing returns property is uniformly bounded. Section 5 presents very encouraging computational experiments for large scale bandit problems drawn from a generative family of coin type bandits. In the interest of implementability, Section 6 develops a combinatorial algorithm for the fast computation of packing heuristic policies for multi-armed bandits. Section 7 concludes with a perspective on interesting directions for future work.

\section{Model}

We consider a multi-armed bandit problem with multiple simultaneous
`pulls' permitted at every time step. A single  bandit arm (indexed
by $i$) is a Markov Decision Process (MDP) specified by a state
space $\Sscr_i$, an action space, $\Ascr_i$, a reward function $r_i:
\Sscr_i \times \Ascr_i \rightarrow \R_+$, and a transition kernel
$P_i: \Sscr_i \times \Ascr_i \rightarrow \Delta_{\Sscr_i}$ (where
$\Delta_{\Sscr_i}$ is the $|\Sscr_i|$-dimensional unit simplex),
yielding a probability distribution over next states should one
choose some action $a_i \in \Ascr_i$ in state $s_i \in \Sscr_i$.

Every bandit arm is endowed with a distinguished `idle' action
$\phi_i$. Should a bandit be idled in some time period, it
yields no rewards in that period and transitions to the same state
with probability $1$ in the next period. More precisely,
\[
\begin{split}
r_i(s_i, \phi_i) & = 0, \ \ \forall s_i \in \Sscr_i,
\\
P_i(s_i, \phi_i, s_i) & = 1,
 \ \ \forall s_i \in \Sscr_i.
\end{split}
\]

We consider a bandit problem with $n$ arms. In each time step one
must select a subset of up to $k (\leq n)$ arms for which one may pick any
action available at those respective arms. Should an action other
than the idle action be selected at any of these $k$ arms, we refer
to such a selection as a `pull' of that arm. That is, any action $a_i
\in \Ascr_i \setminus \{\phi_i \}$ would be considered a pull of the
$i$th arm. One is forced to pick the idle action for the remaining
$n-k$ arms. We wish to find an action selection (or control) policy
that maximizes expected rewards earned over $T$ time periods. Our
problem may be cast as an optimal control problem. In particular, we
define as our \emph{state-space} the set $\Sscr = \prod_i \Sscr_i$
and as our \emph{action space}, the set $\Ascr = \prod_i \Ascr_i$.
We let $\Tscr = \{0,1,\dots, T-1\}$. We understand by $s_i$, the
$i$th component of $s \in \Sscr$ and similarly let $a_i$ denote the $i$th component of $a \in \Ascr$. A feasible action is
one which calls for simultaneously pulling at most $k$ arms. In
particular we let $\Ascr^{\rm feas} = \{a \in \Ascr, \sum_i
\mathbf{1}_{a_i \neq \phi_i} \leq k \}$ denote the set of all
feasible actions. We define a reward function $r: \Sscr \times \Ascr
\rightarrow \R_+$, given by $r(s,a) = \sum_i r_i(s_i,a_i)$ and a
system transition kernel $P: \Sscr \times \Ascr \rightarrow
\Delta_{\prod_i \Sscr_i}$, given by $P(s,a,s') = \Pi_i P_i(s_i,a_i,s'_i)$.


We now formally develop what we mean by a control policy. The arm
selection policy we will eventually develop will use auxiliary
information aside from the current state of the system, and so we
require a general definition. Let $X_0$ be a random variable
 that
encapsulates any endogenous randomization in selecting an action,
and define the filtration generated by $X_0$ and the history of
visited states and actions by
$$\Fscr_t = \sigma(X_0, (s^0),(s^1,a^0),\dots,(s^t,a^{t-1})),$$
where $s^t$ and $a^t$ denote the state and action at time $t$,
respectively. We assume that $\mathbb{P}(s^{t+1} = s'| s^t = s, a^t
= a, H_t = h_t) = P(s,a,s')$ for all $s,s' \in \Sscr, a \in \Ascr, t
\in \Tscr$ and any $\Fscr_t$-measurable random variable $H_t$. A
\emph{feasible} policy simply specifies a sequence of $\Ascr^{\rm feas}$-valued actions $\{a^t\}$
adapted to $\Fscr_t$. In particular, such a policy may be specified by a collection of $\sigma(X^0)$ measurable, $\Ascr^{\rm feas}$-valued random variables, $\{ \mu(s^0,\dots,s^t,a_0,\dots,a^{t-1},t)\}$, one for each possible state-action history of the system. We let $M$ denote the set of all such policies $\mu$, and denote by $J^{\mu}(s,0)$ the expected value of using policy $\mu$ starting in state $s$ at time $0$; in particular
$$J^\mu(s,0) = E\left[\sum_{t=0}^{T-1} R(s^{t},a^{t})|s^{0} = s\right],$$
where $a^{t} = \mu(s^0,\dots,s^{t},a^0,\dots,a^{t-1},t)$.

Our goal is to compute an optimal admissible policy. Markovian
policies, i.e. policies under which $a^t$ is measurable with respect
to $\sigma(X^0,s^t)$, are particularly useful. A Markovian policy is specified as a collection of independent $\Ascr^{\rm feas}$ valued random variables $\{ \mu(s,t) \}$ each measurable with respect to $\sigma(X_0)$.
%
%
In particular, assuming the system is in state $s$
at time $t$, such a policy selects an action $a^t$ as the random variable $\mu(s,t)$, independent of past states and actions.  We let $M^{\rm m}$ denote the set of all such admissible Markovian policies.

Every $\mu \in M^{\rm m}$ is associated with a value function,
$J^{\mu}: \Sscr \times \Tscr \rightarrow \R^+$ which, for every
$(s,t) \in \Sscr \times \Tscr$,  gives the expected value of using
control policy $\mu$ starting at that state:
$$J^\mu(s,t) = E\left[\sum_{t'=t}^{T-1} R(s^{t'},\mu(s^{t'},t'))|s^{t} = s\right].$$
We denote by $J^*$ the optimal value function. In particular,
$J^*(s,t) = \sup_{\mu \in M^{\rm m}}J^\mu(s,t)$. The preceding supremum is always achieved and we denote by
$\mu^*$ a corresponding optimal Markovian control policy. That is,
$\mu^* \in \argsup_{\mu \in M^{\rm m}} J^{\mu}(s,t)$ for all
$(s,t) \in \Sscr \times \Tscr$. Our restriction to
Markovian policies is without loss; $M^{\rm m}$ always contains an optimal
policy among the broader class of admissible policies so that
$\sup_{\mu \in M^{\rm m}}J^\mu(s,0) = \sup_{\mu \in M}J^\mu(s,0)$ for all states $s$.
We next formulate a mathematical program to compute such an optimal
policy.

\subsection{Computing an Optimal Policy}
An optimal policy $\mu^*$ may be found via the solution of the
following linear program, $LP(\tilde{\pi}_0)$, specified by a
parameter $\tilde{\pi}_0 \in  \Delta_{\Sscr}$ that
specifies the distribution of arm states at time $t=0$.

$$\begin{array}{lll}
\max. \quad & \sum_t \sum_{s,a}  \pi(s,a,t)R(s,a),       \\
\st \quad & \sum_a \pi(s,a,t) =
\sum_{s',a'}P(s',a',s)\pi(s',a',t-1),
& \forall t > 0, s \in \Sscr,\\
\quad & \pi(s,a,t) = 0,
& \forall s,t, a \notin \Ascr^{\rm feas} \\
\quad & \sum_{a} \pi(s,a,0) = \tilde{\pi}_0(s),
& \forall s \in \Sscr,\\
\quad & \pi \geq 0.
\end{array}$$
where the variables are the state action frequencies $\pi(s,a,t)$,
which give the probability of being in state $s$ at time $t$ and
choosing action $a$. The first set of constraints in the above
program simply enforce the dynamics of the system, while the second
set of constraints enforces the requirement that at most
$k$ arms are simultaneously pulled at any point in time.

An optimal solution to the program above may be used to construct a
policy $\mu^*$ that attains expected value $J^*(s,0)$ starting at
any state $s$ for which $\tilde{\pi}_0(s) > 0$. In particular, given an optimal
solution $\pi^{\rm opt}$ to $LP(\tilde{\pi}_0)$, one obtains such a
policy by defining $\mu^*(s,t)$ as a random
variable that takes value $a \in \Ascr$ with probability $\pi^{\rm
opt}(s,a,t)/\sum_a \pi^{\rm opt}(s,a,t)$.  By construction, we have
$E[J^*(s,0)|s \sim \tilde{\pi}_0] = OPT(LP(\tilde{\pi}_0))$. Of
course, efficient solution of the above program is not a tractable
task, which forces us to seek approximations to an optimal policy.
The next section will present one such policy with an appealing structural property we term `irrevocability'.

\section{An Irrevocable Approximation to the Optimal Policy}

This section develops an approximation to the optimal multi-armed bandit control policy that we will
subsequently establish performs adequately relative to the optimal
policy. This approximation will possess a desirable property we term `irrevocability'. In particular, the policy we develop will, at any time, be permitted to pull an arm only if that arm was pulled in the prior time step, or else never pulled in the past.

We first develop a control policy for a related
bandit problem, where the requirement that precisely $k$ arms be
pulled in any time step is relaxed. As we will see, this is
essentially Whittle's relaxation and the policy developed for this relaxation is an upper bound to the optimal policy.
We will then use the control policy developed for this relaxed control problem to design a
policy for the multi-armed bandit problem that is irrevocable and
also offers good performance relative to the optimal
policy for a broad class of bandits.

Consider the following relaxation of the program $LP(\tilde{\pi}_0)$, $RLP(\tilde{\pi}_0)$.
$RLP(\tilde{\pi}_0)$ may be viewed as a primal formulation of
Whittle's relaxation:

$$\begin{array}{lll}
\max. \quad & \sum_i \sum_t \sum_{s_i,a_i}  \pi_i(s_i,a_i,t) r_i(s_i,a_i),       \\
\st \quad & \sum_{a_i} \pi_i(s_i,a_i,t) =
\sum_{s'_i,a'_i}P_i(s_i',a_i',s_i)\pi_i(s_i',a_i',t-1),
& \forall t > 0, s_i \in \Sscr_i, i.\\
\quad & \sum_i \left[ T - \sum_{s_i} \sum_t \pi_i(s_i,\phi_i,t)
\right] \leq kT,
\\
\quad & \sum_{a_i} \pi_i(s_i,a_i,0) = \sum_{\bar{s}: \bar{s}_i = s_i} \tilde{\pi}_0(\bar{s}),
\\
\quad & \pi \geq 0,
\end{array}$$
where $\pi_i(s_i,a_i,t)$ is the probability of the $i$th bandit being in
state $s_i$ at time $t$ and choosing action $a_i$.

The program above relaxes the requirement that precisely $k$ arms be
pulled in a given time step; instead we now require that over the
entire horizon at most $kT$ arms are pulled in \emph{expectation},
where the expectation is over policy randomization and state
evolution. The first set of equality constraints enforce individual
arm dynamics whereas the first inequality constraint enforces the
requirement that at most $kT$ arms be pulled in expectation  over
the entire time horizon. The following lemma makes the notion of a relaxation to $LP(\tilde{\pi}_0)$ precise; the proof may be found in the appendix.

\begin{lemma}
\label{le:relaxation}
$OPT(RLP(\tilde{\pi}_0)) \geq OPT(LP(\tilde{\pi}_0))$
\end{lemma}

 Given an optimal solution $\bar{\pi}$ to $RLP(\tilde{\pi}_0)$, one may consider the policy $\mu^{R}$,
 that, assuming we are in state $s$ at time $t$, selects a random action $\mu^{R}(s,t)$,
 where $\mu^{R}(s,t) = a$ with probability
 $\prod_i \left(\bar{\pi}_i(s_i,a_i,t)/\sum_{a_i} \bar{\pi}_i(s_i,a_i,t)\right)$ independent of the
 past. Noting that the action for each arm $i$ is chosen independently of all other arms,
 we use $\mu^{R}_i(s_i,t)$ to denote the induced policy for arm $i$. By construction,
 $E[J^{\mu^{R}}(s,0)|s \sim \tilde{\pi}_0] = OPT(RLP(\tilde{\pi}_0))$. Moreover, we have that $\mu^R$ satisfies the constraint
 $$E^{\mu}\left[\sum_{t=0}^{T-1} \sum_i \mathbf{1}_{\mu^R_i(s^t_i,t) \neq \phi_i}|s^0 = s\right] \leq kT,$$
where the expectation is over random state transitions and
endogenous policy randomization.

Of course, $\mu^{R}$ is not necessarily feasible; we ultimately require a policy that entails at most $k$ arm pulls in any time step. We will use $\mu^{R}$ to construct such a feasible policy. In addition, we will see that if an arm is pulled and then idled in some subsequent time step, it will never again be pulled, so that the policy we construct will be irrevocable.
In what follows we will assume for convenience
that $\tilde{\pi}_0$ is degenerate and puts mass $1$ on a single
starting state. That is, $\tilde{\pi}_0(s_i) = 1$ for some $s_i \in
\Sscr_i$ for all $i$. We first introduce some relevant notation.
Given an optimal solution $\bar{\pi}$ to $RLP(\tilde{\pi}_0)$,
define the value generated by arm $i$ as the random variable
\[
R_i = \sum_{t=0}^{T-1} r_i(s^t_i,\mu^{R}_i(s^t_i,t)),
\]
and the `active time' of arm $i$, $T_i$ as the total number of pulls of arm
$i$ entailed under that policy
\[
T_i = \sum_{t=0}^{T-1} 1_{\mu^{R}_i(s^t_i,t) \neq \phi_i}.
\]
The expected value of arm $i$, $E[R_i] = \sum_{s_i,a_i,t}
\bar{\pi}_i(s_i,a_i,t)r_i(s_i,a_i)$, and the expected active time
$E[T_i] = \sum_{s_i,a_i,t: a_i \neq \phi_i} \bar{\pi}_i(s_i,a_i,t)$.
We will assume in what follows that $E[T_i] > 0$ for all $i$;
otherwise, we simply consider eliminating those $i$ for which
$E[T_i] = 0$. We will also assume for analytical convenience that
$\sum_i E[T_i] = kT$. Neither assumption results in a loss of
generality.

To motivate our policy we begin with the following analogy with a packing
problem: Imagine packing $n$ objects into a knapsack of size $B$.
Each object $i$ has size $\tilde{T}_i$ and value $\tilde{R}_i$.
Moreover, we assume that we are allowed to pack fractional
quantities of an object into the knapsack and that packing a
fraction $\alpha$ of the $i$th object requires space $\alpha \tilde{T}_i$
and generates value $\alpha \tilde{R}_i$. An optimal policy is then
given by the following greedy procedure: select objects in
decreasing order of the ratio $\tilde{R}_i/\tilde{T}_i$ and place
them in to the knapsack to the extent that there is room available.
If one had more than a single knapsack and the additional constraint
that an item could not be placed in more than a single knapsack,
then the situation is more complicated. One may consider a greedy procedure that, as before, considers items in decreasing order of the ratio $\tilde{R}_i/\tilde{T}_i$ and places them (possibly fractionally) in sequence, into the least loaded of the bins at that point.  This generalization of the greedy procedure for the simple knapsack is suboptimal, but still a reasonable heuristic.

Thus motivated, we begin with a loose high level description of our
control policy, which we call the `packing' heuristic. We think of each bandit arm $i$ as an `item' of value $E[R_i]$ with size $E[T_i]$. For the purposes of this explanation \emph{alone}, we will
assume for convenience that should policy $\mu^{R}$ call for an
arm that was pulled in the past to be idled, it will never again call for that arm to be pulled; we will momentarily remove that assumption.
Our control policy will operate as follows: we will order arms in
decreasing order of the ratio $E[R_i]/E[T_i]$. We begin with the top
$k$ arms according to this ordering. For each such arm we will
select an action according to the policy specified for that arm by
$\mu_i^{R}$; should this policy call for the arm to be idled, we
discard that arm and will never again consider pulling it. We
replace the discarded arm with the next available arm (in order of
initial arm rankings) and select an action for the arm according to
$\mu^{R}$. We repeat this procedure until we have selected
non-idle actions for up to $k$ arms (or no arms are available). We
then let time advance, earn rewards, and repeat the procedure
described above until the end of the time horizon.

Algorithm \ref{alg:kMAB} describes the packing heuristic policy precisely, addressing the fact that $\mu_i^{R}$ may call for an arm to be idled but then pulled in some subsequent time step.

\begin{algorithm}
\caption{The Packing Heuristic} \label{alg:kMAB}
\begin{algorithmic}[1]
\STATE
Renumber bandits so that $\frac{E[{R}_1]}{E[{T}_1]} \geq
\frac{E[{R}_2]}{E[{T}_2]} \dots \geq \frac{E[{R}_N]}{E[{T}_N]}$.
Index bandits by variable $i$.
\STATE
$l_i \gets 0, a_i \gets
\phi_i$ for all $i$, $s \sim \tilde{\pi}_0(\cdot)$
\newline
\COMMENT
{The `local time' of every arm is set to $0$ and its
designated action to the idle action. An initial state is drawn according to the initial state distribution $\tilde{\pi}_0$.}
\STATE
$J \GETS 0$
\COMMENT
{Total reward earned is initialized to $0$.}
\STATE
$\mathbb{X} \gets \{1,2,\dots,k  \}, \mathbb{A} \GETS \{k+1,\dots,n\},
\mathbb{D} = \emptyset$.
\newline
\COMMENT
{Initialize the set of active ($\Xb$), available ($\Ab$), and discarded ($\Db$) arms.}
\FOR
{$t=0$ to $T-1$}
\WHILE
[Select up to $k$ arms to pull.]
{there exists an arm $i \in \Xb$ with $a_i = \phi_i$}
\STATE
Select an $i \in \Xb$ with
$a_i=\phi_i$
\newline
\COMMENT
{In what follows, either select an action for arm $i$ or else discard it.}
\WHILE
[Attempt to select a pull action for arm $i$]
{$a_i = \phi_i$ and $l_i < T$}
\STATE
Select $a_i \propto \bar{\pi}_i(s_i,\cdot,l_i)$
\COMMENT
{Select an action according to the solution to $RLP(\bar{\pi})$.}
\STATE
$l_i \gets \l_i +1$
\COMMENT
{Increment arm $i$'s local time.}
\ENDWHILE
\IF
[Discard arm $i$ and activate next highest ranked arm available.]
{$l_i = T$ and $a_i = \phi_i$}
\STATE
$\Xb \gets \Xb \setminus \{i\}, \Db \gets \Db \cup \{i\}$
\COMMENT {Discard arm $i$.}
\IF[There are available arms.]
{$\Ab \neq \emptyset$}
\STATE
$j \GETS \min \Ab$
\COMMENT {Select highest ranked available arm.}
\STATE
$\Xb \GETS \Xb \cup \{j\}, \Ab \GETS \Ab \setminus \{j\}$
\COMMENT
{Add arm to active set.}
\ENDIF
\ENDIF
\ENDWHILE
\FOR
[Pull selected arms.]
{Every $i \in \Xb$}
\STATE
$s_i \sim P(s_i,a_i,\cdot)$
\newline
\COMMENT
{Pull arm $i$; select next arm $i$ state according to its transition
kernel assuming the use of action $a_i$.}
\STATE
$J \gets J + r_i(s_i,a_i)$
\COMMENT
{Earn rewards.}
\STATE
$a_i \GETS \phi_i$
\ENDFOR
\ENDFOR
\end{algorithmic}
\end{algorithm}

In the event that we placed no restriction on the time horizon (i.e.
we set $T= \infty$ in the algorithm above), we have by construction,
that the expected total reward earned under the above policy is
precisely $OPT(RLP(\tilde{\pi}_0))$. In essence,
$RLP(\tilde{\pi}_0)$ prescribes a policy wherein each arm generates
a total reward with mean $E[{R}_i]$ using an expected total number
of pulls $E[{T}_i]$, independent of other arms. The above scheme may
be visualized as one which `packs' as many of the pulls of various
arms possible in a manner so as to meet feasibility constraints.

It is clear that the heuristic we have constructed entails a minimal amount of arm `exploration'. In particular, we are guaranteed at most $n-k$ changes to the set of pulled arms. One may naturally ask what the limited exploration permitted under this policy costs us in terms of performance. In addition, is this scheme computationally practical? In particular, the linear programming relaxation we must solve is still a fairly large program. In subsequent sections we address these issues. First, we present a theoretical analysis that demonstrates that the price of irrevocability is uniformly bounded for an important general class of bandits. Our analysis sheds light on the structural properties that are likely to afford a low price of irrevocability in practice. We then present results of computational experiments with a generative family of large-scale problems demonstrating performance losses of up to $5-10 \%$ percent relative to an upper bound on the performance of the optimal policy (which is potentially non-irreovcable and has no restrictions on exploration). Finally, we address computational issues relevant to the packing heuristic and develop a computational scheme that is substantially quicker than heuristics such as Whittle's heuristic.

\section{The Price of Irrevocability}

This section establishes a uniform bound on the performance loss
incurred in using the irrevocable packing heuristic relative to an optimal, potentially non-irrevocable scheme
for a useful family of bandits whose arms exhibit a certain
decreasing returns property. This class includes bandits whose arms
are coins of unknown biases -- a family particularly relevant to a
number of applications including those discussed in the
introduction. We establish that the packing heuristic always earns
expected rewards that are within a factor of $1/8$ of an optimal
scheme. Our analysis sheds light on those structural properties that likely afford a low price to irrevocability. In addition
to being an indicator of robustness across all
parameter regimes, this bound on the price of irrevocability is remarkable for two reasons. First,
it does not rely on an asymptotic scaling of the system; the performance of the packing heuristic will `track' that of an optimal, potentially non-irrevocable heuristic across all regimes. Second, the
bound represents a comparison with a system where one is allowed
recourse to arms that were pulled in the past and discarded. In
particular, the bound thus highlights the fact that for a useful
class of bandits, one may achieve reasonable performance with very
limited exploration. The typical performance we expect from the heuristic is likely to be far superior (as it generally is in the case of problems for which such worst case guarantees can be established); in a subsequent section we will present computational experiments indicating a performance loss of $5-10$\% relative to an optimal policy with no restrictions on exploration.

In what follows we first specify the decreasing returns property and
explicitly identify a class of bandits that possess this property.
We then present our performance analysis which will proceed as follows:  we first consider pulling
bandit arms \emph{serially}, i.e. at most one arm at a time, in
order of their rank and show that the total reward earned from
bandits that were first pulled within the first $kT/2$ pulls is at
least within a factor of $1/8$ of an optimal policy. This result
relies on the static ranking of bandit arms used, and a
symmetrization idea exploited by \cite{GoemansDean00} in their result on
stochastic packing where rewards are statistically independent of
item size. In contrast to that work, we must address the fact that the rewards
earned from a bandit are statistically dependent on the number of
pulls of that bandit and to this end we exploit the decreasing
returns property that establishes the nature of this correlation. We
then show via a combinatorial sample path argument that the expected reward earned
from bandits pulled within the first $T/2$ time steps of the packing
heuristic i.e., with arms being pulled in parallel, is at least as
much as that earned in the setting above where arms are pulled
serially, thereby establishing our performance guarantee.

\subsection{The Decreasing Returns Property}

Define for every $i$ and $l < T$, the random variable
$$L_i(l) = \sum_{t=0}^{l} \mathbf{1}_{\mu^{R}_i(s_i^{t},t)  \neq \phi_i }.$$
$L_i(l)$ tracks the number of times a given arm $i$ has been pulled
under policy $\mu^{R}$ among the first $l+1$ steps of selecting an
action for that arm. Further, define
$$R_i^m =  \sum_{l=0}^{T-1} 1_{L_i(l) \leq m}  r_i(s_i^{l},\mu^{R}_i(s_i^{l},l)).$$
$R_i^m$ is the random reward earned within the first $m$ pulls of
arm $i$ under the policy $\mu^{R}$.   The decreasing returns
property roughly states that the expected incremental returns from
allowing an additional pull of a bandit arm are, on average,
decreasing. More precisely, we have:

\begin{property}(Decreasing Returns)\label{as:dec_returns}
$E[R_i^{m+1}] - E[R_i^{m}] \leq E[R_i^m] - E[R_i^{m-1}]$ for all $0
< m<T$.
\end{property}

One useful class of bandits from a modeling perspective that satisfy
this property are bandits whose arms are `coins' of unknown bias.
The following discussion makes this notion more precise:

\subsubsection{An example of a bandit with decreasing returns: Coins}

We define a `coin' to be any multi-armed bandit for which every arm
$i$ has action space $a_i = \{p, \phi_i\}$, with $r(s_i,p) > 0$ for
all $s_i \in \Sscr_i$, and satisfies the following property:
\[
r(s_i,p) \geq \sum_{s'_i \in \Sscr_i} P(s_i,p,s'_i)r(s'_i,p),
\quad \forall s_i \in \Sscr_i.
\]
The above sub-martingale characterization of rewards intuitively suggests the decreasing returns property. In particular, it suggests that the returns from a pull in the current state are at least as large as the expected returns to a pull in a state reached subsequent to the current pull. The decreasing returns property for coins is established in the following Lemma whose proof may be found in the appendix:

\begin{lemma}
\label{le:coin_dec_returns}
Coins satisfy the decreasing returns property. That is, if $\Ascr_i = \{p,\phi_i\} \ \forall i$, and
\[
r(s_i,p) \geq \sum_{s'_i \in \Sscr_i} P(s_i,p,s'_i)r(s'_i,p),
\quad \forall i, s_i \in \Sscr_i,
\]
then
\[
E[R_i^{m+1}] - E[R_i^{m}] \leq E[R_i^m] - E[R_i^{m-1}]
\]
for all $0< m<T$.
\end{lemma}
Returning to our motivating example of dynamic product assortment selection, we note that in estimating the bias of a binomial coin of unknown bias given some initial prior on coin bias, Bayes' rule implies that the estimated bias after $n$ observations (which generate the filtration $\Fscr_n$), $\mu_{n+1}$ satisfies $E[\mu_{n+1}|\Fscr_n] = \mu_n$. Thus, bandits with such arms wherein the reward from an arm is some non-negative scalar times the bias, automatically possess the decreasing returns property.

\subsection{A Uniform Bound on the Price of Irrevocability for Bandits with `Decreasing Returns'}

For convenience of exposition we assume that $T$ is even; addressing the odd case requires essentially identical proofs but cumbersome notation.

We re-order the bandits in decreasing order of $E[R_i]/E[T_i]$ as in
the packing heuristic. Let us define
$$H^* = \min \left\{j: \sum_{i=1}^{j} E[T_i] \geq kT/2\right\}.$$
Thus, $H^*$ is the set of bandits that take up approximately half
the budget on total expected pulls. Next, let us define for all $i$,
random variables $\tilde{R}_i$ and $\tilde{T}_i$ according to
$\tilde{R}_i = R_i, \tilde{T}_i = T_i$ for all $i < H^*$. We define
$\tilde{R}_{H^*} = \alpha R_{H^*}$ and $\tilde{T}_{H^*} = \alpha
T_{H^*}$, where $\alpha = \frac{kT/2 -
\sum_{i=1}^{H^*-1}E[T_i]}{E[T_{H^*}]}$.

We begin with a preliminary lemma:

\begin{lemma}
\label{le:concave}
$$\sum_{i=1}^{H^*} E[\tilde{R}_i] \geq \frac{1}{2} OPT(RLP(\tilde{\pi}_0)).$$
\end{lemma}
\begin{proof}
Define a function
$$f(t) = \sum_{i=1}^n \frac{E[R_i]}{E[T_i]} \left(\left(t  - \sum_{j=1}^{i-1} E[T_i]\right)^+ \wedge E[T_i]\right),$$
where $(a\wedge b) = \min (a,b)$. By construction (i.e. since
$\frac{E[R_i]}{E[T_i]}$ is non-increasing in $i$), we have that $f$
is a concave function on $[0,kT].$ Now observe that
$$\sum_{i=1}^{H^*} E[\tilde{R}_i]
= \sum_{i=1}^{H^*-1} \frac{E[R_i]}{E[T_i]}E[T_i] +
\frac{E[R_{H^*}]}{E[T_{H^*}]}\left(kT/2 - \sum_{j =
1}^{H^*-1}E[T_i]\right) = f(kT/2).$$ Next, observe that
$$ OPT(RLP(\tilde{\pi}_0)) = \sum_{i=1}^n \frac{E[R_i]}{E[T_i]} E[T_i] = f(kT).$$
By the concavity of $f$ and since $f(0)=0$, we have that $f(kT/2)
\geq \frac{1}{2}f(kT)$, which yields the result.
\end{proof}

We next compare the expected reward earned by a certain subset of
bandits with indices no larger than $H^*$. The significance of the
subset of bandits we define will be seen later in the proof of Lemma
\ref{le:simulation} -- we will see there that all bandits in this
subset will begin operation prior to time $T/2$ in a run the packing heuristic. In particular, define
$$R_{1/2} = \sum_{i=1}^{H^*}\mathbf{1}_{\{ \sum_{j=1}^{i-1}T_j < kT/2 \}}R_i.$$

\begin{lemma}
\label{le:half_inequality}
$$E[R_{1/2}] \geq \frac{1}{4} OPT(RLP(\tilde{\pi}_0)).$$
\end{lemma}
\begin{proof}
We have:
\[
\begin{split}
E[R_{1/2}] &\stackrel{(a)} = \sum_{i=1}^{H^*}{\rm Pr}\left(\sum
_{j=1}^{i-1} T_j < kT/2\right)E[R_i]
\\
&\stackrel{(b)} \geq \sum_{i=1}^{H^*}{\rm Pr}\left(\sum _{j=1}^{i-1}
T_j < kT/2\right)E[\tilde{R}_i]
\\
&\stackrel{(c)} = \sum_{i=1}^{H^*}{\rm Pr}\left(\sum _{j=1}^{i-1}
\tilde{T}_j < kT/2\right)E[\tilde{R}_i]
\\
& \stackrel{(d)} \geq \sum_{i=1}^{H^*} \left(1 -
\frac{\sum_{j=1}^{i-1} E[\tilde{T}_j]}{kT/2}\right) E[\tilde{R}_i]
\\
& = \sum_{i=1}^{H^*}E[\tilde{R}_i]  -
\sum_{i=1}^{H^*}\frac{\sum_{j=1}^{i-1} E[\tilde{T}_j] }{kT/2}
E[\tilde{R}_i]
\\
&\stackrel{(e)} \geq \sum_{i=1}^{H^*}E[\tilde{R}_i]  -
\frac{1}{2}\sum_{i=1}^{H^*}\frac{\sum_{j=1, j \neq i}^{H^*}
E[\tilde{T}_j] }{kT/2} E[\tilde{R}_i]
\\
& \stackrel{(f)}\geq \frac{1}{2} \sum_{i=1}^{H^*}E[\tilde{R}_i]
\\
& \stackrel{(g)}\geq \frac{1}{4} OPT(RLP(\tilde{\pi}_0))
\end{split}
\]
Equality~(a) follows from the fact that under policy $\mu^{R}$,
$R_i$ is independent of $T_j$ for $j < i$. Inequality~(b) follows
from our definition of $\tilde{R}_i$: $\tilde{R}_i \leq R_i$.
Equality~(c) follows from the fact that by definition $\tilde{T}_i =
T_i$ for all $i < H^*$. Inequality~(d) invokes Markov's inequality.

Inequality~(e) is the critical step in establishing the result and
uses the simple symmetrization idea exploited by \cite{GoemansDean00}: In
particular, we observe that since $\frac{E[R_i]}{E[T_i]} \leq
\frac{E[R_j]}{E[T_j]}$ for $i > j$, it follows that $E[R_i]E[T_j]
\leq \frac{1}{2} (E[R_i]E[T_j] + E[R_j]E[T_i])$ for $i > j$. Replacing every
term of the form $E[R_i]E[T_j]$ (with $i > j$) in the expression
preceding inequality~(e) with the upper bound $\frac{1}{2}
(E[R_i]E[T_j] + E[R_j]E[T_i])$ yields inequality~(e).
Inequality~(f) follows from the fact that
$\sum_{i=1}^{H^*}E[\tilde{T}_i] = kT/2$. Inequality~(g) follows from
Lemma \ref{le:concave}.
\end{proof}

Before moving on to our main Lemma that translates the above
guarantees to a guarantee on the performance of the packing heuristic, we need to establish one additional technical fact.
Recall that $R_i^m$ is the reward earned by bandit $i$ in the first $m$
pulls of this bandit. Exploiting the assumed decreasing returns property, we have the following Lemma whose proof may be found in the appendix:

\begin{lemma}
\label{le:fact1}
For bandits satisfying the decreasing returns property (Property \ref{as:dec_returns}),
\[
E\left[\sum_{i=1}^{H^*}\mathbf{1}_{\sum_{j=1}^{i-1}T_j < kT/2 }
R_i^{T/2}\right] \geq \frac{1}{2}E[R_{1/2}].
\]
\end{lemma}

We have thus far established estimates for total expected rewards
earned assuming implicitly that bandits are pulled in a serial
fashion in order of their rank. The following Lemma connects these
estimates to the expected reward earned under the $\mu^{\rm packing}$
policy (given by the packing heuristic) using a simple sample path
argument. In particular, the following Lemma shows that the expected
rewards under the $\mu^{\rm packing}$ policy are at least as large as
$E\left[\sum_{i=1}^{H^*}\mathbf{1}_{\sum_{j=1}^{i-1}T_j < kT/2 }
R_i^{T/2}\right]$.

\begin{lemma}
\label{le:simulation} $E[J^{\mu^{\rm packing}}(s,0)| s\sim \tilde{\pi}_0]
\geq E\left[\sum_{i=1}^{H^*}\mathbf{1}_{\sum_{j=1}^{i-1}T_j < kT/2 }
R_i^{T/2}\right]. $
\end{lemma}
\begin{proof}
For a given sample path of the system define
$$h = (H^*) \wedge \min \left\{i: \sum_{j=1}^{i}T_j \geq kT/2 \right\}.$$
On this sample path, it must be that:
\begin{equation}
\label{eq:spath_equality}
\sum_{i=1}^{H^*}\mathbf{1}_{\sum_{j=1}^{i-1}T_j < kT/2 } R_i^{T/2} = \sum_{i=1}^h R_i^{T/2}.
\end{equation}

We claim that arms $1,2,\dots,h$ are all first pulled at times $t <
T/2$ under $\mu^{\rm packing}$. Assume to the contrary that this were not
the case and recall that arms are considered in order of index under
$\mu^{\rm packing}$, so that an arm with index $i$ is pulled for the first
time no later than the first time arm $l$ is pulled for $l > i$. Let
$h'$ be the highest arm index among the arms pulled at time
$t=T/2-1$ so that $h' < h$. It must be that $\sum_{j=1}^{h'} T_j
\geq kT/2$. But then,
\[
H^* \wedge \min \left\{i: \sum_{j=1}^{i}T_j \geq kT/2 \right\} \leq
h'
\]
which is a contradiction.

Thus, since every one of the arms $1,2,\dots,h$ is first pulled at
times $t < T/2$, each such arm may be pulled for at least $T/2$ time
steps prior to time $T$ (the horizon). Consequently, we have that the
total rewards earned on this sample path under policy $\mu^{\rm packing}$
are at least
\[
\sum_{i=1}^h R_i^{T/2}
\]
Using identity \eqref{eq:spath_equality} and taking an expectation over sample paths yields the result.
\end{proof}

We are ready to establish our main Theorem that
provides a uniform bound on the performance loss incurred in using
the packing heuristic policy relative to an optimal policy with no restrictions on exploration. In particular, we have that the price of irrevocability is uniformly bounded for bandits satisfying the decreasing returns property.

\begin{theorem}
\label{th:unif_bound} For multi-armed bandits satisfying the
decreasing returns property (Property \ref{as:dec_returns}), $
E[J^{\mu^{\rm packing}}(s,0)| s\sim \tilde{\pi}_0] \geq \frac{1}{8}
E[J^{*}(s,0)| s\sim \tilde{\pi}_0] $ for all initial state
distributions $\tilde{\pi}_0$.
\end{theorem}
\begin{proof}
We have from Lemmas \ref{le:relaxation},\ref{le:half_inequality},\ref{le:fact1} and
\ref{le:simulation} that
\[
E[J^{\mu^{\rm packing}}(s,0)| s\sim \tilde{\pi}_0] \geq \frac{1}{8}OPT(RLP(\tilde{\pi}_0)).
\]
We know from Lemma \ref{le:relaxation} that $OPT(RLP(\tilde{\pi}_0)) \geq OPT(LP(\tilde{\pi}_0)) = E[J^{*}(s,0)| s\sim \tilde{\pi}_0]$ from which the result follows.
\end{proof}

Our analysis highlighted a structural property -- decreasing returns -- that is likely to afford a low price of irrevocability. The next section demonstrates computational results that suggest that in practice we may expect this price to be quite small (on the order of $5-10$\%) for bandits possessing this property.

\section{Computational Experiments}

This section presents computational experiments with the packing heuristic. We consider a number of large scale bandit problems drawn from a generative family of problems to be discussed shortly and demonstrate that the packing heuristic consistently demonstrates performance within about $5-10$ \% of an upper bound on the performance of an unrestricted (i.e. potentially \emph{non}-irrevocable) optimal solution to the multi-armed bandit problem. In particular, this suggests that the price of irrevocability is likely to be small in practice, at least for models of the type we consider here. Since the bandits considered in our experiments - Binomial coins of uncertain bias - are among the most widely used applications of the multi-armed bandit model, we view this to be a positive result.


\textbf{The Generative Model: } We consider multi-armed bandit problems with $n$ arms up to $k$ of which may be pulled simultaneously at any time. The $i$th arm corresponds to a Binomial$(m,P_i)$ coin where $m$ is fixed and known, and $P_i$ is unknown but drawn from a Dirichlet$(\alpha_i,\beta_i)$ prior distribution. Assuming we choose to `pull' arm $i$ at some point, we realize a random outcome $M_i \in \{0,1,\dots,m\}$. $M_i$ is a Bernoulli$(m,P_i)$ random variable where $P_i$ is itself a Dirichlet$(\alpha_i,\beta_i)$ random variable. We receive a reward of $r_iM_i$ and update the prior distribution parameters according to $\alpha_i \leftarrow \alpha_i+M_i$, $\beta_i \leftarrow \beta_i + m - M_i$. By selecting the initial values of  $\alpha_i$ and $\beta_i$  for each arm appropriately we can control for the initial uncertainty in the value of $P_i$. This model is, for instance, applicable to the dynamic assortment selection problem discussed earlier (see \cite{Gallien07}) with each coin representing a product of uncertain popularity and $M_i$ representing the uncertain number of product $i$ sales over a single period in which that product is offered for sale. We recall from our previous discussion that this family of bandits satisfies the decreasing returns property and from our performance analysis we expect a reasonable price of irrevocability.

We consider the following random instances of the above problem. We consider bandits with $(n,k) \in \{(500,50), (500,100), (100,10), (100,20)\}$. These dimensions are representative of large scale applications of which the dynamic assortment problem is an example. For each value of $(n,k)$ we consider time horizons $T=25$ and $T=40$. For every bandit problem we consider, we subdivide the arms of the bandit into $10$ groups. All arms within a group have identical statistical structure, that is, identical $r_i$ values and identical initial values of $\alpha_i$ and $\beta_i$. For each value of $(n,k,T)$, we generate a number of problem instances by randomly drawing prior parameters for bandit arms.  In particular, for all arms in a given group we select $\alpha_i$ uniformly in the interval $[0.05,0.35]$ and then select that value of $\beta_i$ which results in a prior co-efficient of variation $cv \in \{1,2.5\}$. These co-efficients of variation represent, respectively, a moderate and high degree of a-priori uncertainty in coin bias (or in the context of the dynamic assortment application, product popularity). In addition, $r_i$ is drawn uniformly on $[0,2]$ and we take $m=2$. We generate $100$ random problem instances for each co-efficient of variation. Control policies for a given bandit problem instance are evaluated over $3000$ random state trajectories (which resulted in $98$\% confidence intervals that were at least within  +/-$1$\% of the sample average).

\begin{table}
\begin{center}
\label{tab:performance}
\begin{tabular}{|c|c|c|c|c|}
\hline
\multicolumn{5}{|c|}{Summary of Computational Experiments} \\
\hline
Coeff. of Variation & Arms & Simultaneous Pulls &  Horizon  & Performance \\
$(cv)$ & $(n)$ & $(k)$ & $(T)$ & $(J^{\mu^{\rm packing}}/J^*)$\\ \hline
        & 500 & 50 & 25  & 0.91 \\
        & 500 & 50 & 40 & 0.92 \\
        & 500 & 100 & 25& 0.93\\
 Moderate & 500 & 100 & 40 &0.94\\
 $(1)$  & 100 & 10 & 25 & 0.88\\
        & 100 & 10 & 40 & 0.89\\
        & 100 & 20 & 25 & 0.91\\
        & 100 & 20 & 40 & 0.93\\ \hline
        & 500 & 50 & 25  & 0.89\\
        & 500 & 50 & 40 & 0.90\\
        & 500 & 100 & 25& 0.90\\
 High    & 500 & 100 & 40 &0.91\\
 $(2.5)$    & 100 & 10 & 25 & 0.87\\
        & 100 & 10 & 40 & 0.88\\
        & 100 & 20 & 25 & 0.90\\
        & 100 & 20 & 40 & 0.91\\ \hline
\end{tabular}
\caption{Computational Summary.  Each row represents summary statistics for 100 distinct random bandit problems with the specified $n,k,T$ and $cv$ parameters. Performance for each instance was computed from $3000$ simulations of that instance. Performance figures thus represent an average over the generative family with the specified $n,k,T$ and $cv$ parameters as also over system randomness.}
\end{center}
\end{table}

\textbf{Evaluating Performance: } A striking feature of our performance results is that the price of irrevocability is quite small, a trend that appears to hold over varying parameter regimes. In particular, we make the following observation:
\begin{itemize}
\item Consider problems with a small number of arms ($100$) with a large number of simultaneous pulls ($20$) allowed. Intuitively, an optimal policy could reasonably explore all arms in this setting before settling on the `best' arms. We thus expect the price of irrevocability to be high here. Even in this regime we find that the price of irrevocability is only about $10-11$ \% of optimal performance.
\item Consider problems with a high degree of a-priori uncertainty in coin bias. Mistakes - that is, discarding an arm that is performing reasonably in favor of an unexplored arm that turns out to perform poorly - are particularly expensive in such problems. With a hgh co-efficient of variation in the prior on initial arm bias, the price of ir-revocability is indeed somewhat higher but continues to remain within $10-12$ \% of optimal performance.
\item For each of our experiments, we observe that keeping all other parameters fixed, relative performance improves with a longer time horizon. This is intuitive; with longer horizons, one may delay discarding an arm only once one is sure that the arm performs poorly relative to the expected value of the available alternatives.
\item Finally, we note that the performance figures we report are relative to an \emph{upper bound} on optimal policy performance. Computing the optimal policy is itself an intractable task. The performance observed here suggests that at least for bandit problems with decreasing returns the packing heuristic is a viable approximation scheme even when irrevocability is not necessarily a concern.
\end{itemize}

We can thus conclude that the price of irrevocability is small for a
useful class of multi-armed bandit problems and that the packing
heuristic performs well for this class of problems. A final concern
is computational effort. In particular, for the largest problem
instance we considered ($n=500$), the linear program we need to
solve has $3.2$ million variables and about the same number of constraints.
Even a commercial linear programming solver (such as
CPLEX) equipped with the ability to exploit structure in this
program will require several hours on a powerful computer to solve
this program. This is in stark contrast with an index based
heuristic (such as Whittles heuristic) that solves a simple dynamic
program for each arm at every time step. In the next section we
develop an efficient computational algorithm for the solution of
$RLP(\tilde{\pi}_0)$ that requires substantially less effort than
even Whittles heuristic and takes a few minutes to solve the
aforementioned program on a laptop computer.

\section{Fast Computation}

This section considers the computational effort required to implement the packing heuristic. We develop a computational scheme that makes the packing heuristic substantially easier to implement than popular index heuristics such as Whittle's heuristic and thus establish that the heuristic is viable from a computational perspective.

The key computational step in implementing the packing heuristic is
the solution of the linear program $RLP(\tilde{\pi}_0)$. Assuming
that $|\Sscr_i| = O(S)$ and $|\Ascr_i| = O(A)$ for all $i$, this
linear program has $O(nTAS)$ variables and each Newton iteration of
a general purpose interior point method will require
$O\left((nTAS)^3\right)$ steps. An interior point method that
exploits the fact that bandit arms are coupled via a single
constraint will require $O(n(TAS)^3)$ computational steps at each
iteration. We develop a combinatorial scheme to solve this linear
program that is in spirit similar to the classical Dantzig-Wolfe
dual decomposition algorithm. In contrast with Dantzig-Wolfe
decomposition, our scheme is efficient. In particular, the scheme
requires $O(nTAS^2 \log (kT))$ computational steps to solve
$RLP(\tilde{\pi}_0)$ making it a significantly faster solution
alternative to the schemes alluded to above. Equipped with this fast
scheme, it is notable that using the packing heuristic requires
$O(nAS^2 \log (kT))$ computations per time step amortized over the
time horizon which will typically be substantially less than the $\Theta(nAS^2T)$
computations required per time step for index policy heuristics such
as Whittle's heuristic.

Our scheme employs a `dual decomposition' of $RLP(\tilde{\pi}_0)$. The key technical difficulty we must overcome in developing our computational scheme for the solution of $RLP(\tilde{\pi}_0)$ is the non-differentiability of the dual function corresponding to $RLP(\tilde{\pi}_0)$ at an optimal dual solution
which prevents us from recovering an optimal or near optimal policy by direct minimization of the dual function.

\subsection{An Overview of the Scheme}

For each bandit arm $i$, define the polytope $D_i(\tilde{\pi}_0) \in
\mathbb{R}^{|\Sscr_i||\Ascr_i|T}$ of permissible state-action
frequencies for that bandit arm specified via the constraints of
$RLP(\tilde{\pi}_0)$ relevant to that arm.


A point within this polytope, $\pi_i$, corresponds to a set of valid
state-action frequencies for the $i$th bandit arm. With some abuse of notation, we denote the
expected reward from this arm under $\pi_i$ by the `value' function:
\[ R_i(\pi_i) = \sum_{t=0}^{T-1} \pi_i(s_i,a_i,t)r_i(s_i,a_i).\]
In addition denote the expected number of pulls of bandit arm $i$
under $\pi_i$ by
\[
T_i(\pi_i) = T - \sum_{s_i} \sum_t \pi_i(s_i,\phi_i,t).
\]
We understand that both $R_i(\cdot)$ and $T_i(\cdot)$ are defined
over the domain  $D_i(\tilde{\pi}_0)$.

We may thus rewrite $RLP(\tilde{\pi}_0)$ in the following form:

\begin{equation}
\label{eqn:mod_rlp}
\begin{array}{lll}
\max. \quad & \sum_i R_i(\pi_i),     \\
\st \quad & \sum_i T_i(\pi_i) \leq kT.
\end{array}
\end{equation}

The Lagrangian dual of this program is $DRLP(\tilde{\pi}_0)$:

$$\begin{array}{lll}
\min. \quad & \lambda kT + \sum_i \max_{\pi_i} \left(R_i(\pi_i) -
\lambda T_i(\pi_i)\right),
\\
\st \quad & \lambda\geq 0.
\\
\end{array}$$

The above program is convex. In particular, the objective is a
convex function of $\lambda$. We will show that strong duality
applies to the dual pair of programs above, so that the optimal
solution to the two programs have identical value. Next, we will observe
that for a given value of $\lambda$, it is simple to compute
$\max_{\pi_i}  \left(R_i(\pi_i) - \lambda T_i(\pi_i)\right)$ via the
solution of a dynamic program over the state space of arm $i$ (a
fast procedure). Finally it is simple to derive useful a-priori
lower and upper bounds on the optimal dual solution $\lambda^*$.
Thus, in order to solve the dual program, one may simply employ a
bisection search over $\lambda$. Since for a given value of
$\lambda$, the objective may be evaluated via the solution of $n$
simple dynamic programs, the overall procedure of solving the dual program $DRLP(\tilde{\pi}_0)$ is fast.

What we ultimately require is the optimal solution to the primal program $RLP(\tilde{\pi}_0)$. One natural way we might hope to do this (that ultimately will not work) is the following: Having
computed an optimal dual solution $\lambda^*$, one may hope to
recover an optimal primal solution, $\pi^*$ (which is what we
ultimately want), via the solution of the problem
\begin{equation}
\label{eq:dualopt} \max_{\pi_i}  \left(R_i(\pi_i) - \lambda^*
T_i(\pi_i)\right).
\end{equation}
for each $i$. This is the typical dual decomposition procedure. Unfortunately, this last step \emph{need not} necessarily
yield a feasible solution to $RLP(\tilde{\pi}_0)$. In particular, solving
\eqref{eq:dualopt} for $\lambda = \lambda^*+\epsilon$ may result in
an arbitrarily suboptimal solution for any $\epsilon >0$, while
solving  \eqref{eq:dualopt} for a $\lambda \leq \lambda^*$ may yield
an infeasible solution to $RLP(\tilde{\pi}_0)$.
The technical reason for this is that the Lagrangian dual function for $RLP(\tilde{\pi})$ may
be non-differentiable at $\lambda^*$. These difficulties are far from pathological, and Example \ref{eg}
illustrates how they may arise in a very simple example.

\begin{example}
\label{eg} The following example illustrates that the
dual function may be non-differentiable at an optimal solution, and that it is not sufficient
to solve \eqref{eq:dualopt} for $\lambda \leq \lambda^*$ or $\lambda
= \lambda^*+\epsilon$ for an $\epsilon>0$ arbitrarily small.
Specifically, consider the case where we have $n=2$ identical bandits,
$T=1$, and $K=1$. Each bandit starts in state $s$, and two actions
can be chosen for it, namely, $a$ and the idling action $\phi$. The
rewards are $r(s,a) = 1$ and $r(s,\phi)=0$.  Thus,
$RLP(\tilde{\pi}_0)$ for this specific case is given by:
$$\begin{array}{lll}
\max. \quad & \pi_1(s,a,0) + \pi_2(s,a,0),
\\
\st \quad & \pi_1(s,a,0) + \pi_2(s,a,0)\leq 1,
\\
\end{array}$$
where $\pi_i\in D_i(\tilde{\pi}_0)$, $i=1,2$. Clearly, the optimal
objective function value for the above optimization problem is 1.
The Lagrangian dual function for the above problem is
$$
\begin{aligned}
g(\lambda) &= \lambda + \max_{\pi_1(s,a,0)}\pi_1(s,a,0)(1-\lambda) +
\max_{\pi_2(s,a,0)}\pi_2(s,a,0)(1-\lambda)\\
           &=\left \{
           \begin{array}{ll}
            2 - \lambda & \lambda \leq 1\\
            \lambda & \lambda > 1
           \end{array}
           \right.
\end{aligned}$$
Not the dual function is minimized at $\lambda^*=1$, which is a point of non-differentiability. Moreover,
solving \eqref{eq:dualopt} at $\lambda^* + \epsilon$ for any
$\epsilon>0$, gives $\pi_1(s,a,0) = \pi_2(s,a,0)=0$ which is clearly
suboptimal. Also, a solution for $0\leq \lambda \leq \lambda^*$ is
$\pi_1(s,a,0) = \pi_2(s,a,0)=1$, which is clearly infeasible.

\end{example}

Notice that in the above example, the average of the solutions to problem \eqref{eq:dualopt} for $\lambda = \lambda^* - \epsilon$ and $\lambda = \lambda^* + \epsilon$ \emph{does} yield a feasible, optimal primal solution, $\pi_1(s,a,0) = \pi_2(s,a,0)= 1/2$. We overcome the difficulties presented by the non-differentiability of the dual function by computing both upper and lower
approximations to $\lambda^*$, and computing solutions to
\eqref{eq:dualopt} for both of these approximations. We then
consider as our candidate solution to $RLP(\tilde{\pi}_0)$, a
certain convex combination of the two solutions. In particular, we
propose algorithm \ref{alg:solver}, that takes as input the
specification of the bandit and a tolerance parameter $\epsilon$.
The algorithm produces a feasible solution to $RLP(\tilde{\pi}_0)$
that is within an additive factor of $2\epsilon$ of optimal.


\begin{algorithm}
\caption{RLP SOLVER} \label{alg:solver}
\begin{algorithmic}[1]
\STATE $\lambda^{\rm feas} \gets r_{\rm max}+\delta,\text{ for any }
\delta>0,  \lambda^{\rm infeas} \gets 0$. \STATE For all $i$,
$\pi^{\rm feas}_i \gets \pi_i \in \argmax_{\pi_i} \left(R_i(\pi_i) -
\lambda^{\rm feas} T_i(\pi_i)\right)$,
\\\qquad \qquad $\pi^{\rm infeas}_i \gets
\pi_i\in\argmax_{\pi_i} \left(R_i(\pi_i) - \lambda^{\rm infeas}
T_i(\pi_i)\right)$. \WHILE { $\lambda^{\rm feas} - \lambda^{\rm
infeas}
> \frac{\epsilon}{kT}$}
\STATE $\lambda \gets \frac{\lambda^{\rm feas} + \lambda^{\rm
infeas}}{2}$ \FOR {$i=1$ to $n$} \STATE $\pi^*_i \gets \pi_i \in
\argmax_{\pi_i}  \left(R_i(\pi_i) - \lambda T_i(\pi_i)\right)$.
\ENDFOR \IF {$\sum_{i=1}^n T(\pi^*_i) > kT$} \STATE $\lambda^{\rm
infeas} \gets \lambda, \pi^{\rm infeas}_i \GETS \pi^*_i, \ \forall
i$ \ELSE \STATE $\lambda^{\rm feas} \gets \lambda, \pi^{\rm feas}_i
\GETS \pi^*_i, \ \forall i$ \ENDIF \ENDWHILE \IF
{$\sum_iT_i(\pi^{\rm infeas}_i) - T_i(\pi^{\rm feas}_i)>0$} \STATE
$\alpha \gets \frac {kT - \sum_iT_i(\pi_i^{\rm feas}) }
{\sum_iT_i(\pi^{\rm infeas}_i) - T_i(\pi^{\rm feas}_i)} \wedge 1 $
\ELSE \STATE $\alpha \gets 0$ \ENDIF \FOR {$i=1$ to $n$} \STATE
$\pi^{\rm RLP}_i \gets \alpha \pi^{\rm infeas}_i + (1-\alpha)
\pi^{\rm feas}_i$ \ENDFOR
\end{algorithmic}
\end{algorithm}

It is clear that the bisection search above will require $O(\log
(r_{\rm max} kT/\epsilon))$ steps (where $r_{\rm max} = \max_{i,s_i,a_i} r(s_i,a_i)$).
At each step in this search, we solve $n$
problems of the type in \eqref{eq:dualopt}, i.e. $\max_{\pi_i} \left(R_i(\pi_i) - \lambda
T_i(\pi_i)\right)$. These subproblems may be reduced to a dynamic
program over the state space of a single arm. In particular, we
define a reward function $\tilde{r}_i:\Sscr_i \rightarrow
\mathbb{R}_+$ according to $\tilde{r}_i(s_i,a_i) = r_i(s_i,a_i) -
\lambda 1_{a_i \neq \phi_i}$ and compute the value of an optimal
policy starting at state $s_{0,i}$ (where $s_0$ is that state on
which $\tilde{\pi}_0$ places mass $1$) assuming $\tilde{r}_i$ as the
reward function. This requires $O(S^2AT)$ steps per arm. Thus the
RLP Solver algorithm requires a total of $O(nS^2AT\log r_{\rm max} kT/\epsilon)$
computational steps prior to termination. The following theorem,
proved in the appendix establishes the quality of the solution
produced by the RLP Solver algorithm:

\begin{theorem}
\label{th:solver_correctness}
RLP Solver produce a feasible solution to $RLP(\tilde{\pi}_0)$ of
value at least $OPT(RLP(\tilde{\pi}_0)) - 2\epsilon$.
\end{theorem}

The RLP Solver scheme was used for all computational experiments in
the previous section. Using this scheme, the largest problem instances we considered were
solved in a few minutes on a laptop computer.

\section{Concluding Remarks}

This paper introduced `irrevocable' policies for the multi-armed bandit problem. We hope to draw two main conclusions from our presentation thus far. First, in addition to being a desirable constraint to impose on a multi-armed bandit policy, irrevocability is frequently a \emph{cheap} and thereby practical constraint. In particular, we have attempted to show via a host of computational results as also a theoretical analysis that the price of irrevocability, i.e. the performance loss incurred relative to an optimal scheme with no restrictions on exploration, is likely to be small in practice. In fact, the performance of the heuristic we developed suggests that it is likely to be useful even when irrevocability is not a concern. Second, we have shown that computing good irrevocable policies for multi-armed bandit problems is easy. In particular, we developed a fast computational scheme to accomplish this task. This scheme is faster than a widely used heuristic for general multi-armed bandit problems.

This research serves as a point of departure for a number of interesting questions that we believe would be interesting to explore:

\begin{itemize}
\item The packing heuristic is one of many possible irrevocable heuristics. It is attractive since it offers satisfactory performance in computational experiments, affords a worst-case price of irrevocability analysis (and so is theoretically robust), and finally can be implemented with less computational efforts than typical index heuristics. The packing heuristic is, however, by no means the only irrevocable heuristic one may construct. An alternative for instance, would be to modify Whittle's heuristic so as to simply ignore bandits that were discarded in the past or equivalently formulate a corresponding restless bandit problem where discarded arms generate no rewards. The present work establishes irrevocability as a \emph{cheap} constraint to place on many multi-armed bandit problems. Moving forward, one may hope to construct other high performing irrevocable heuristics for the multi-armed bandit problem.
\item
What happens to the price of irrevocability in interesting asymptotic parameter regimes? An example of such a regime may include for instance, simultaneously scaling the number of bandits, the number of simultaneous plays allowed, $k$, as also the horizon $T$. The correlation in the reward earned from a bandit arm and the length of time the arm is pulled preclude a useful, straightforward large deviations type extension to our analysis. Nonetheless, this is an important question to ask.
\item
We have illustrated the decreasing returns property for coins. This property is fairly natural though, and there are a number of other types of bandits that may well possess this property. Staying in the vain of dynamic assortment selection, it may well be the case that incorporating inventory and pricing decisions for each product may still yield bandits satisfying this property.
\item
In addition to bandits satisfying the decreasing returns property, are there other interesting classes of bandits that afford a low price of irrevocability?
\item Is irrevocability a cheap constraint for interesting classes of \emph{restless} bandit problems? Given our results, it is intuitive to expect that this may be the case for bandits with switching costs  which represent one simple class of restless bandit problems.
\end{itemize}

Our work thus far suggests that irrevocable policies are an effective means of extending the practical applicability of multi-armed bandit approaches in several interesting scenarios such as dynamic assortment problems or sequential drug trials where recourse to drugs that were discontinued from trials at some point in the past is socially unacceptable. Progress on the questions above will likely further the goal of extending the practical applicability of the multi-armed bandit approach.

\small{
\bibliographystyle{ormsv080} 
\bibliography{MAB-bib}
}

\begin{appendix}
\section{Proofs for Section 3}
\begin{l1}
$OPT(RLP(\tilde{\pi}_0)) \geq OPT(LP(\tilde{\pi}_0))$
\end{l1}
\begin{proof}
Let $\hat{\pi}$ be an optimal solution to $LP(\tilde{\pi}_0)$. We construct a feasible solution to $RLP(\tilde{\pi}_0)$ of equal value. In particular, define a candidate solution to $RLP(\tilde{\pi}_0)$, $\bar{\pi}$ according to
\[
\bar{\pi}(s_i, a_i, t) = \sum_{\bar{s},\bar{a}:\bar{s}_i = s_i, \bar{a}_i = a_i}   \hat{\pi}(\bar{s},\bar{a},t)
\]
This solution has value precisely $OPT(LP(\tilde{\pi}_0))$. It remains to establish feasibility. For this we first observe that
\begin{equation}
\label{eq:con1}
\begin{split}
\sum_{s'_i,a'_i}P_i(s_i',a_i',s_i)\bar{\pi}_i(s_i',a_i',t-1)
&
=
\sum_{s'_i,a'_i}P_i(s_i',a_i',s_i) \sum_{\bar{s},\bar{a}:\bar{s}_i = s'_i, \bar{a}_i = a'_i}   \hat{\pi}(\bar{s},\bar{a},t-1)
\\
&
=
\sum_{s'_i,a'_i}P_i(s_i',a_i',s_i)
\sum_{\bar{s},\bar{a}:\bar{s}_i = s'_i, \bar{a}_i = a'_i}
\left(
\sum_{\hat{s}_{-i}}
\prod_{j \neq i} P(\bar{s}_j,\bar{a}_j,\hat{s}_j)
\right)
\hat{\pi}(\bar{s},\bar{a},t-1)
\\
&
=
\sum_{s',a',\hat{s}:\hat{s}_i = s_i} P(s',a',\hat{s})\hat{\pi}(s',a',t-1)
\\
&
=
\sum_{\hat{s}: \hat{s}_i = s_i,  a} \hat{\pi}(\hat{s},a,t)
\\
&
=
\sum_{a_i} \bar{\pi}_i(s_i,a_i,t)
\end{split}
\end{equation}

Next, we observe that the expected total number of pulls of arm pulls under the policy prescribed by $\hat{\pi}$ is simply
\[
\sum_i \sum_{s,t,a:a_i \neq \phi_i} \hat{\pi}(s,a,t)
\]
Since the total number of pulls in a given time step under $\hat{\pi}$ is at most $k$, we have
\[
\sum_i \sum_{s,t,a:a_i \neq \phi_i} \hat{\pi}(s,a,t) \leq kT
\]
But,
\[
\begin{split}
\sum_i \sum_{s,t,a:a_i \neq \phi_i} \hat{\pi}(s,a,t)
& =
\sum_i \sum_t \sum_{s_i, a_i \neq \phi_i} \bar{\pi}(s_i,a_i,t)
\\
&
=
\sum_i \sum_t \left( 1 - \sum_{s_i} \bar{\pi}(s_i,\phi_i,t)  \right)
\\
&
=
\sum_i \left( T - \sum_{s_i,t} \bar{\pi}(s_i,\phi_i,t)\right),
\end{split}
\]
so that
\begin{equation}
\label{eq:con2}
\sum_i \left( T - \sum_{s_i,t} \bar{\pi}(s_i,\phi_i,t)\right) \leq kT
\end{equation}

From \eqref{eq:con1} and \eqref{eq:con2},  $\bar{\pi}$ is indeed a feasible solution to $RLP(\tilde{\pi}_0)$. This completes the proof.
\end{proof}

\section{Proofs for Section 4}

\begin{l2}
Coins satisfy the decreasing returns property. That is, if $\Ascr_i = \{p,\phi_i\} \ \forall i$, and
\[
r(s_i,p) \geq \sum_{s'_i \in \Sscr_i} P(s_i,p,s'_i)r(s'_i,p),
\quad \forall i, s_i \in \Sscr_i,
\]
then
\[
E[R_i^{m+1}] - E[R_i^{m}] \leq E[R_i^m] - E[R_i^{m-1}]
\]
for all $0< m<T$.
\end{l2}
\begin{proof}
To see that coins satisfy the decreasing returns property, we
first introduce some notation. It is clear that the policy
$\mu^{R}_i$ induces a Markov process on the state space $\Sscr_i$.
We expand this state space, so as to track the total number of arm
pulls so that our state space now become $\Sscr_i \times \{ \Tscr, T\}$. The
policy $\mu^{R}_i$ induces a distribution over arm $i$ states for
every time $t <T$, which we denote by the variable $\hat{\pi}$.
Thus, $\hat{\pi}(s_i,m,t,a_i)$ will denote the probability of being in
state $(s_i,m)$ at time $t$ and taking action $a_i$.

Now,
$$E[R_i^{m+1} - R_i^{m}] = \sum_{s,t<T} \hat{\pi}(s_i,m,t,p)r(s,p)$$
and similarly, for $E[R_i^{m} - R_i^{m-1}]$.

But,
\[
\begin{split}
\sum_{s_i,t<T} \hat{\pi}(s_i,m,t,p)r(s_i,p) & = \sum_{s_i,t<T-1}
\hat{\pi}(s_i,m-1,t,p) \left( \sum_{s_i'} P(s_i,p,s_i')g(s_i',t+1)r(s_i',p)
\right)
\\
& \leq \sum_{s_i,t<T-1} \hat{\pi}(s_i,m-1,t,p) \left( \sum_{s_i'}
P(s_i,p,s_i')r(s_i',p) \right)
\\
& \leq \sum_{s_i,t<T-1} \hat{\pi}(s_i,m-1,t,p)r(s_i,p)
\\
& \leq \sum_{s_i,t<T} \hat{\pi}(s_i,m-1,t,p)r(s_i,p)
\\
& = E[R_i^{m} - R_i^{m-1}]
\end{split}
\]
where $g(s_i,t) = 1- \prod_{t'=t}^{T-1}{\rm Pr}(\mu^{R}_i(s_i,t')
= \phi_i)$. Here $\prod_{t'=t}^{T-1}{\rm Pr}(\mu^{R}_i(s_i,t') =
\phi_i)$ is the probability of never pulling the arm after reaching
state $s_i$ at time $t$ so that $g(s_i,t)$ represents the probability of eventually pulling arm $i$ after reaching state $s_i$ at time $t$.
The second inequality follows from the assumption on reward structure in the statement of the Lemma. We thus see that coins satisfy the decreasing returns property.
\end{proof}

\begin{l5}
For bandits satisfying the decreasing returns property (Property \ref{as:dec_returns}),
\[
E\left[\sum_{i=1}^{H^*}\mathbf{1}_{\sum_{j=1}^{i-1}T_j < kT/2 }
R_i^{T/2}\right] \geq \frac{1}{2}E[R_{1/2}].
\]
\end{l5}
\begin{proof}

We note that assuming Property \ref{as:dec_returns} implies that $E[R_i^{T/2}]
\geq \frac{1}{2}E[R_i]$ for all $i$. The assertion of the Lemma is
then evident -- in particular,
\[
\begin{split}
E[R_{1/2}] & = \sum_{i=1}^{H^*}{\rm Pr}\left(\sum _{j=1}^{i-1} T_j <
kT/2\right)E[R_i]
\\
& \leq \sum_{i=1}^{H^*}{\rm Pr}\left(\sum _{j=1}^{i-1} T_j <
kT/2\right)2E[R_i^{T/2}]
\\
& = 2 E\left[\sum_{i=1}^{H^*}\mathbf{1}_{\sum_{j=1}^{i-1}T_j < kT/2
} R_i^{T/2}\right]
\end{split}
\]
where the first and second equality use the fact that $R_i$ and $R_i^{T/2}$ are each independent of $T_j$ for $j \neq i$.
\end{proof}

\section{Proof of Theorem \ref{th:solver_correctness}}

The following lemma shows that the optimal objective function value
of the dual is equal to $OPT(RLP(\tilde{\pi}_0))$. In particular, it
shows that Slater's constraint qualification condition holds~(see,
for example, \cite{Boyd04}).
\begin{lemma}
\label{lem:slater}
 $OPT(RLP(\tilde{\pi}_0))=
OPT(DRLP(\tilde{\pi}_0))$. That is, strong duality holds.
\end{lemma}
\begin{proof}
To show this, it is sufficient to show that there is a strictly
feasible solution to~(\ref{eqn:mod_rlp}), i.e., the inequality is
satisfied strictly. This is straightforward -- in particular, for
each bandit $i$, set $\pi_i(s_i,\phi_i,t) = \tilde{\pi}_{0,i}(s_i)$
for all $s_i$ and $t$,  where $\tilde{\pi}_{0,i}(s_i)$ is the
probability of bandit $i$ starting in state $s_i$. Set
$\pi_i(s_i,a_i,t)=0$ for $a_i\neq \phi_i$ for all $s_i,t$.  These
state action frequencies belong to $D_i(\tilde{\pi}_0)$, and also
give $T_i(\pi_i)=0$.
\end{proof}

We denote $R^* = OPT(RLP(\tilde{\pi}_0))= OPT(DRLP(\tilde{\pi}_0))$.
Also, define the following set of total running times for all
bandits corresponding to a dual variable $\lambda$:
$$\mathcal{T}(\lambda) = \left\{ \sum_{i}T_i(\pi_i) \left \vert \pi_i \in
\argmax_{\pi_i}\left(R_i(\pi_i)-\lambda T_i(\pi_i) \right) \right),~
\forall i\right\}.$$

\begin{lemma}
\label{lem:monotone}
 If $0\leq\lambda_1 < \lambda_2$, then
\[ \min \mathcal{T}(\lambda_1) \geq \max \mathcal{T}(\lambda_2) .\]
\end{lemma}
\begin{proof}
We denote the objective function in $DRLP(\tilde{\pi}_0)$, i.e., the
dual function by:
\[g(\lambda) =  \lambda kT + \sum_i\max_{\pi_i}\left(R_i(\pi_i) - \lambda T_i(\pi_i)
    \right).\]
 The slack in the total running time constraint
$\sum_{i}T_i(\pi_i)\leq kT$, i.e. $kT - \sum_i T(\pi_i)$, is a subgradient of $g$ for any $\pi$
such that $\pi_i\in \argmax_{\pi_i}\left(R_i(\pi_i)-\lambda
T_i(\pi_i)\right)$~(see \cite{Shor85}). Thus, the set of
subgradients of the dual function $g$ at $\lambda$ are given by
\[ \partial g(\lambda) = \{k T - t   : t\in \mathcal{T}(\lambda)\}.\]
Then, since $g$ is a convex function, it follows that for
$0\leq\lambda_1 < \lambda_2$,
\[kT - t_1 \leq kT - t_2, \qquad \forall ~ t_1 \in \mathcal{T}(\lambda_1),~t_2 \in \mathcal{T}(\lambda_2).\]
The lemma then follows.
\end{proof}

$(\pi^*,\lambda^*)$ is an optimal solution for the primal and dual
problems if and only if~(see, for example, \cite{Boyd04})
\begin{equation}
\label{eqn:opt}
\begin{aligned}
&\pi_i^* \in \argmax_{\pi_i}\left(R_i(\pi_i)-\lambda^* T_i(\pi_i)\right ),\\
&\text{either}\quad\lambda^*>0 \text{ and }\sum_{i}T_i(\pi^*_i) =
kT, \quad \text{or} \quad \lambda^* = 0 \text{ and }
\sum_{i}T_i(\pi^*_i) \leq kT.
\end{aligned}
\end{equation}

We prove the correctness of the $RLP$ Solver algorithm separately
for the cases when $\lambda^*=0$ is optimal, and when any optimal
solution satisfies $\lambda^*>0$. We denote the values of the bounds
on the dual variable that are computed by the \emph{last iteration}
of the $RLP$ solver algorithm by $\lambda^{\rm feas}$ and
$\lambda^{\rm infeas}$. Recall that,
$$
\begin{aligned}
\pi_i^{\rm feas} &\in \argmax_{\pi_i}\left( R_i(\pi_i) -
\lambda^{\rm feas}T_i(\pi_i)\right),\\
\pi_i^{\rm infeas} &\in \argmax_{\pi_i}\left( R_i(\pi_i) -
\lambda^{\rm infeas}T_i(\pi_i)\right).
\end{aligned}
$$
We introduce some additional notation:
$$
\begin{aligned}
T^{\rm feas} &= \sum_i T_i(\pi_i^{\rm feas}), \qquad R^{\rm feas} = \sum_i R_i(\pi_i^{\rm feas}), \\
T^{\rm infeas} &= \sum_i T_i(\pi_i^{\rm infeas}),\quad R^{\rm
infeas} = \sum_i R_i(\pi_i^{\rm infeas}).
\end{aligned}
$$
Thus,
\begin{equation}
\label{dual_fn}
\begin{aligned}
g(\lambda^{\rm feas})&=  \lambda^{\rm feas}kT + R^{\rm feas} -
\lambda^{\rm
feas}T^{\rm feas},\\
g(\lambda^{\rm infeas})&=  \lambda^{\rm infeas}kT + R^{\rm infeas} -
\lambda^{\rm
feas}T^{\rm infeas}.\\
\end{aligned}
\end{equation}


\begin{lemma}
\label{lem:zero} If $(\pi^*,\lambda^*)$ is a solution to
\eqref{eqn:opt} with $\lambda^* = 0$, then
$$\smash{R^* - \left(\alpha R^{\rm infeas} +
(1-\alpha)R^{\rm feas}\right) \leq \epsilon.}$$
\end{lemma}
\begin{proof}
 If $\lambda^*=0$, it follows from \eqref{eqn:opt} that there is some $t\in \mathcal{T}(0)$ such that
$t\leq kT$.  Hence, it follows from Lemma~\ref{lem:monotone} that
for any $\lambda>0$, $\max \mathcal{T}(\lambda) \leq kT$. Hence,
Line~11 of the $RLP$ solver algorithm is always invoked, and so, the
$RLP$ solver algorithm converges to
\[\begin{aligned}
\smash{\lambda^{\rm infeas} = 0 \quad \text{and}\quad 0<\lambda^{\rm
feas}<\epsilon/(kT).}
\end{aligned}\]
Hence, $\pi_i^{\rm infeas} \in \argmax_{\pi_i}(R_i(\pi_i))$. Also,
$g(\lambda)$ is minimized at $\lambda^* = 0$. Hence, it follows from
Lemma~\ref{lem:slater} that
\begin{equation}
\label{eqn:Rstar=Rinfeas} \smash{R^* = g(0)= \sum_i \max_{\pi_i}
R_i(\pi_i) = R^{\rm infeas}.}
\end{equation}
Since, $\lambda^{\rm feas}>0$, it follows from
 $T^{\rm feas}\leq kT$. Hence, we now consider the following three cases:
\begin{itemize}
\item {\emph{Case 1:}}
$T^{\rm infeas}\leq kT$.\\ Here, $\alpha =1$, and hence, using
\eqref{eqn:Rstar=Rinfeas} it follows that
$$\smash{R^* - \left(\alpha R^{\rm infeas} +
(1-\alpha)R^{\rm feas}\right) = 0.}$$

\item {\emph{Case 2:}} $T^{\rm feas} = kT$.\\
In this case, $(\pi^{\rm feas}, \lambda^{\rm feas})$ satisfy the
optimality conditions in \eqref{eqn:opt}. Thus, $R^{\rm feas} =
R^*$, and so (since $R^{\rm infeas}= R^*$ by
\eqref{eqn:Rstar=Rinfeas})
$$ R^* - \left(\alpha R^{\rm infeas} + (1-\alpha)R^{\rm feas}\right)
= 0.$$
\item {\emph{Case 3:}}
$T^{\rm infeas}>kT>T^{\rm feas}$.\\
Since, $g(\lambda)$ is minimized at $\lambda = 0$,
\[
\begin{aligned}
R^*= g(0)\leq g(\lambda^{\rm feas})&=  \lambda^{\rm feas}kT + R^{\rm
feas} - \lambda^{\rm
feas}T^{\rm feas}\\
\Rightarrow R^* - R^{\rm feas} &\leq \lambda^{\rm feas}\left(kT -
T^{\rm feas}\right).
\end{aligned}
\]
Since, $R^*=R^{\rm infeas}$ (from \eqref{eqn:Rstar=Rinfeas}), and
using the fact that $0<\alpha<1$ when $T^{\rm infeas}>kT>T^{\rm
feas}$, we have
\begin{equation}
\label{eq:lambda=0}
\begin{split}
R^* - \alpha R^{\rm infeas} - (1-\alpha)R^{\rm feas} &=
(1-\alpha)(R^* - R^{\rm feas})\\
&\leq (1-\alpha)(kT-T^{\rm feas})\lambda^{\rm feas}\\
&
\leq kT \lambda^{\rm feas}\\
&\leq \epsilon,\\
\end{split}
\end{equation}

\end{itemize}
\end{proof}

\begin{lemma}
\label{lem:positive}
If every solution to \eqref{eqn:opt} satisfies
$\lambda^*
> 0$, then
$$R^* - \left(\alpha R^{\rm infeas} + (1-\alpha)R^{\rm feas}\right)
\leq 2\epsilon.$$
\end{lemma}
\begin{proof}
The $RLP$ solver algorithm is initialized with $\lambda^{\rm infeas}
= 0$. Since, $\lambda^*>0$, and $(kT)\in \mathcal{T}(\lambda^*)$
(\eqref{eqn:opt}), it follows from Lemma~\ref{lem:monotone} that
$\min \mathcal{T}(0)\geq kT$. But $(kT)\notin \mathcal{T}(0)$, else
there would be a solution to \eqref{eqn:opt} that satisfies
$\lambda^* = 0$, leading to a contradiction. Thus, $\min
\mathcal{T}(0)> kT$, and so lines 8--12 of the $RLP$ solver
algorithm guarantee that
\begin{equation}
\label{eqn:infeas} T^{\rm infeas}> kT.
\end{equation}
Using an appropriate modification of the optimality conditions in
\eqref{eqn:opt} for the case where the horizon is $T^{\rm infeas}$
(instead of $kT$), we see that $R^{\rm infeas}$ is the maximum
reward earned by any policy in $\{\pi : \sum_{i}T_i(\pi_i)\leq
T^{\rm infeas} \}$.
 Since, $R^*$ is the maximum reward earned by any policy in $\{\pi
: \sum_{i}T_i(\pi_i)\leq kT < T^{\rm infeas}\}$,
\begin{equation}
\label{eqn:infeas2}
 R^{\rm infeas}\geq R^*.
\end{equation}

We now argue that $T^{\rm feas} \leq kT$. The $RLP$ solver algorithm
is initialized with $\lambda^{\rm feas} > r_{\text max}$. Since,
$\pi_i^{\rm feas} \in \argmax_{\pi_i}\left( R_i(\pi_i) -
\lambda^{\rm feas}T_i(\pi_i)\right)$, initially, the optimal policy
is to idle at all times. Thus, $T^{\rm feas} \leq kT$ at
initialization; at all other iterations, lines 8--12 of the
algorithm ensure that $T^{\rm feas} \leq kT$.

We now consider the following two cases separately:

\begin{itemize}
\item{\em{Case 1:}} $T^{\rm feas} = kT$.\\
In this case, $(\pi^{\rm feas}, \lambda^{\rm feas})$ satisfy the
optimality conditions in \eqref{eqn:opt}, and so, $R^{\rm feas} =
R^*$. Now, using \eqref{eqn:infeas2}
\[ \left(\alpha R^{\rm infeas} +
(1-\alpha)R^{\rm feas}\right) \geq R^*.\]

\item{\em{Case 2:}} $T^{\rm feas}< kT$.\\
 Note that the $RLP$ solver
algorithm terminates when
\begin{equation}
\label{eqn:conv1}
\lambda^{\rm feas} - \lambda^{\rm infeas} <
\epsilon/(kT).
\end{equation}

Now $T^{\rm feas}< kT$ and $(kT)\in \mathcal{T}(\lambda^*)$. If
$\lambda^{\rm feas} <   \lambda^*$, it follows from
Lemma~\ref{lem:monotone} that
$$T^{\rm feas} \geq \min \mathcal{T}(\lambda^{\rm feas}) \geq
\max\mathcal{T}(\lambda^*) \geq kT,$$ which is a contradiction.
Hence,
\begin{equation}
\label{eqn:conv2} \lambda^{\rm feas} \geq   \lambda^*.
\end{equation}

Also, since $(kT)\in \mathcal{T}(\lambda^*)$, it follows from
Lemma~\ref{lem:monotone} that for any $\lambda > \lambda^*$,
$\max\mathcal{T}({\lambda}) \leq kT$. So, \eqref{eqn:infeas} implies
that
\begin{equation}
\label{eqn:conv3} \lambda^{\rm infeas} \leq   \lambda^*.
\end{equation}

It follows from
\eqref{eqn:conv1},\eqref{eqn:conv2},\eqref{eqn:conv3} that
\[ \max\left(0,\lambda^* - \frac{\epsilon}{kT}\right) \leq \lambda^{\rm infeas} \quad
\text{and}\quad  \lambda^{\rm feas} \leq \lambda^* +
\frac{\epsilon}{kT}.\]



Since $g(\lambda)$ is minimized at $\lambda^*$, it follows from
\eqref{dual_fn} and strong duality proved in Lemma~\ref{lem:slater}
that
\[\begin{aligned}
g(\lambda^*) = R^*   &\leq g(\lambda^{\rm feas}) = R^{\rm feas} +
\lambda^{\rm feas}\left(kT - T^{\rm feas}\right) \leq R^{\rm
feas}+(\lambda^* + \delta)\left(kT - T^{\rm
feas}\right),\\
g(\lambda^*) = R^* &\leq g(\lambda^{\rm infeas}) =R^{\rm infeas}
 + \lambda^{\rm infeas}\left( kT - T^{\rm infeas}
\right) \leq R^{\rm infeas} +  (\lambda^*-\delta)\left( kT - T^{\rm
infeas} \right),
\end{aligned}\]
where $\delta = \epsilon/(kT)$. Note that the above inequalities
also use $T^{\rm feas} < kT$ (by assumption) and $T^{\rm infeas}>kT$
(from \eqref{eqn:infeas}). Thus,
\[\begin{aligned}
R^* - \alpha R^{\rm infeas} - (1-\alpha)R^{\rm feas} &= \alpha (R^*
-R^{\rm infeas}) + (1-\alpha)(R^*-R^{\rm feas})\\
&\leq \alpha (\delta - \lambda^*)\left( T^{\rm infeas} - kT\right)
+(1-\alpha)(\lambda^* + \delta)\left(kT - T^{\rm feas}\right)
\\
&= 2 \frac{\delta (T^{\rm infeas}-kT)(kT-T^{\rm feas})}{T^{\rm
infeas}- T^{\rm feas}}\\
&\leq 2\delta kT\\
&= 2\epsilon.
\end{aligned}\]
\end{itemize}
\end{proof}

\begin{t2}
RLP Solver produce a feasible solution to $RLP(\tilde{\pi}_0)$ of
value at least $OPT(RLP(\tilde{\pi}_0)) - 2\epsilon$.
\end{t2}
\begin{proof}
The result follows from Lemmas~\ref{lem:zero}
and~\ref{lem:positive}, and the fact that $\lambda^*\geq 0$ (from
\eqref{eqn:opt}).
\end{proof}

\end{appendix}

\end{document}